\documentclass[a4paper,12pt]{article}

\usepackage{amsmath,amssymb,amsthm, mathrsfs}
\usepackage{latexsym}
\usepackage{graphics}
\usepackage{hyperref}
\usepackage{graphicx,psfrag}
\usepackage[bf, small]{titlesec}
\usepackage{authblk} 

\usepackage{lineno}

\theoremstyle{plain}
\newtheorem{Thm}{Theorem}[section]
\newtheorem{Lem}[Thm]{Lemma}
\newtheorem{Prop}[Thm]{Proposition}
\newtheorem{Cor}[Thm]{Corollary}

\theoremstyle{definition}

\usepackage{standalone}
\usepackage{pgfpages,tikz,pgfkeys,pgfplots}
\usetikzlibrary{arrows,positioning,matrix,fit,backgrounds,shapes,shapes.geometric, intersections}

\usetikzlibrary{shapes.callouts,decorations.pathmorphing} 
\usetikzlibrary{shadows} 
\usepackage{calc} 
\usetikzlibrary{decorations.markings} 
\tikzstyle{vertex}=[circle, draw, inner sep=0pt, minimum size=6pt] 

\usetikzlibrary{arrows,matrix} 

\newcommand{\RRR}{\mathcal{R}} 

\newcommand{\WWW}{\mathcal{W}}

\title{On $m$-step competition graphs of bipartite tournaments}

\author[1]{\small Soogang Eoh}
\author[1]{\small Suh-Ryung Kim}
\author[1]{\small Hyesun Yoon}
\affil[1]{\footnotesize Department of Mathematics Education, Seoul National University, Seoul 08826}
\affil[ ]{\footnotesize\textit{mathfish@snu.ac.kr, srkim@snu.ac.kr, magisakura@snu.ac.kr}}
\date{}
\begin{document}
\maketitle
\begin{abstract}
In this paper, we completely characterize the $m$-step competition graph of a bipartite tournament for any integer $m \ge 2$.
In addition, we compute the competition index and the competition period of a bipartite tournament.
\end{abstract}


\section{Introduction}\label{intro}

Given a digraph $D$, the {\em competition graph} $C(D)$ of $D$
has the same vertex set as $D$ and has an edge between vertices $u$ and $v$
if and only if there exists a common prey of $u$ and $v$ in $D$. If $(u,v)$  is an arc of a digraph $D$,
then we call $v$ a {\em prey} of $u$ (in $D$) and call $u$ a {\em predator} of $v$ (in $D$).
The notion of competition graph is due to Cohen~\cite{cohen1968interval} and has arisen from ecology. Competition graphs also have applications in coding, radio transmission, and modeling of complex economic systems.  (See \cite{raychaudhuri1985generalized} and \cite{roberts1999competition} for a summary of these applications.)
Various variants of notion of competition graphs have been introduced and studied (see the survey articles by Kim~\cite{kim1993competition} and Lundgren~\cite{lundgren1989food} for the variations which have been defined and studied by many authors since Cohen introduced the notion of competition graph).

The notion of $m$-step competition graph is one of the important variants and is defined as follows.
Given a digraph $D$ and a positive integer $m$, a vertex $y$ is an {\em $m$-step prey} of a vertex $x$ if and only if there exists a directed walk from $x$ to $y$ of length $m$. Given a digraph $D$ and a positive integer $m$, the digraph $D^m$ has the vertex set same as $D$ and has an arc $(u,v)$ if and only if $v$ is an $m$-step prey of $u$.
Given a positive integer $m$, the {\em $m$-step competition graph} of a digraph $D$, denoted by $C^m(D)$, has the same vertex set as $D$ and has an edge between vertices $u$ and $v$ if and only if there exists an $m$-step common prey of $u$ and $v$. The notion of $m$-step competition graph is introduced by Cho~{\em et al.}~\cite{cho2000m} as a generalization of competition graph.
By definition, it is obvious that $C^1(D)$ for a digraph $D$ is the competition graph $C(D)$.
Since its introduction, it has been extensively studied (see for example \cite{belmont2011complete,cho2011competition,helleloid2005connected,ho2005m,kim2008competition,park2011m,zhao2009note}). Cho~{\em et al.}~\cite{cho2000m} showed that for any digraph $D$ and a positive integer $m$, $C^m(D)=C(D^m)$.

For the two-element Boolean algebra $\mathcal{B}=\{0,1\}$, $\mathcal{B}_n$ denotes the set of all $n \times n$ (Boolean) matrices over $\mathcal{B}$. Under the Boolean operations, we can define matrix
addition and multiplication in $\mathcal{B}_n$.
A graph $G$ is called the {\em row graph} of a matrix $A \in \mathcal{B}_n$ and denoted by $\RRR(A)$ if the rows of $A$ are the vertices of $G$, and two vertices are adjacent in $G$ if and only if their corresponding rows have a nonzero entry in the same column of $A$.
This notion was studied by Greenberg~{\em et al.}~\cite{greenberg1984inverting}. As noted in \cite{greenberg1984inverting}, the competition graph of a digraph $D$ is the row graph of its adjacency matrix.

Cho and Kim~\cite{cho2004competition} introduced the notions of competition index and competition period of $D$ for a strongly connected digraph $D$,
and Kim~\cite{kim2008competition} extended these notions to a general digraph $D$.
Consider the competition graph sequence $C^1(D)$, $C^2(D)$, $C^3(D)$, $\ldots$, $C^m(D)$, $\ldots$ for a digraph $D$.
(Note that for a digraph $D$ and its adjacency matrix $A$, the graph sequence $C^1(D)$, $C^2(D)$, $\ldots$, $C^m(D)$, $\ldots$ is equivalent to the row graph sequence $\mathcal{R}(A)$, $\mathcal{R}(A^2)$, $\ldots$, $\mathcal{R}(A^m)$, $\ldots$.)
Since the cardinality of the Boolean matrix set $\mathcal{B}_n$ is equal to a finite number $2^{n^2}$, there is a smallest positive integer $q$ such that $C^{q+i}(D)=C^{q+r+i}(D)$ (equivalently $\mathcal{R}(A^{q+i})=\mathcal{R}(A^{q+r+i})$) for some positive integer $r$ and all nonnegative integer $i$.
Such an integer $q$ is called the \emph{competition index} of $D$ and is denoted by cindex$(D)$.
For $q=$cindex$(D)$, there is also a smallest positive integer $p$ such that $C^{q}(D)=C^{q+p}(D)$ (equivalently $\mathcal{R}(A^{q})=\mathcal{R}(A^{q+p})$).
Such an integer $p$ is called the \emph{competition period} of $D$ and is denoted by cperiod$(D)$.

Given a graph $G$, let $S \subset V(G)$ be any nonempty subset of vertices of $G$.
The \emph{subgraph of $G$ induced by $S$}, denoted by $G[S]$, is the graph whose vertex set is $S$ and whose edge set consists of all of the edges in $E(G)$
that have both endpoints in $S$.
The same definition works for directed graphs.

In Section~\ref{chap:sink}, we introduce notions of sink elimination index and sink sequence of a digraph and present some useful properties of bipartite tournaments related to $m$-step competition graphs in terms of them.
In Section~\ref{chap:character}, we completely characterize the $m$-step competition graph of a bipartite tournament for any integer $m \ge 2$ and compute the competition index and the competition period of a bipartite tournament.

\section{The sink elimination index and the sink sequence of a digraph}\label{chap:sink}
Given a digraph $D$, we call a vertex of outdegree zero a \emph{sink} in $D$.

We define a nonnegative integer $\zeta(D)$ and sequences
\[(W_0, W_1, \ldots, W_{\zeta(D)}) \quad \mbox{and} \quad (D_0, D_1, \ldots D_{\zeta(D)}) \]
 of subsets of $V(D)$ and subdigraphs of $D$, respectively, as follows.
Let $D_0=D$ and $W_0$ be the set of sinks in $D$.
If $W_0 = V(D)$ or $W_0 = \emptyset$, then let $\zeta(D)=0$.
Otherwise, let $D_1 = D_0-W_0$ and let $W_1$ be the set of sinks in $D_1$.
If $W_1 = V(D_1)$ or $W_1 = \emptyset$, then let $\zeta(D)=1$.
Otherwise, let $D_2 = D_1-W_1$ and let $W_2$ be the set of sinks in $D_2$.
If $W_2 = V(D_2)$ or $W_2 = \emptyset$, then let $\zeta(D)=2$.
We continue in this way until we obtain $W_k=V(D_k)$ or $W_k =\emptyset$ for some nonnegative integer $k$.
Then we let $\zeta(D)=k$.
We call $\zeta(D)$ the \emph{sink elimination index} of $D$ and the sequence $(W_0, W_1, \ldots, W_{\zeta(D)})$ the \emph{sink sequence} of $D$ (see Figure~\ref{fig:sink} for illustration)
and the sequence $(D_0, D_1, \ldots, D_{\zeta(D)})$ the \emph{digraph sequence associated with the sink sequence}.

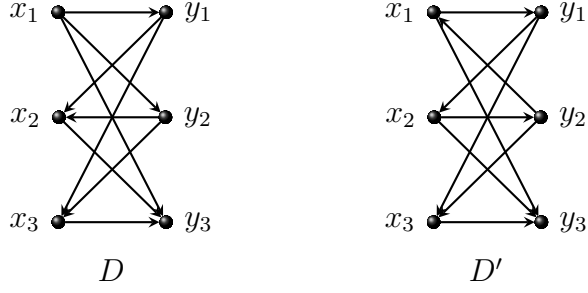
\begin{figure}
  \begin{center}
  \begin{tabular}{ccc}
    \begin{tikzpicture}[auto,thick,scale=0.7]
    \tikzstyle{player}=[minimum size=5pt,inner sep=0pt,outer sep=0pt,ball color=black,circle]
    \tikzstyle{source}=[minimum size=5pt,inner sep=0pt,outer sep=0pt,ball color=black, circle]
    \tikzstyle{arc}=[minimum size=5pt,inner sep=1pt,outer sep=1pt, font=\footnotesize]
    \path (117:2.23cm)  node [player, label=left:$x_1$]  (a) {};
    \path (180:1cm)     node [player, label=left:$x_2$]  (b) {};
    \path (243:2.23cm)  node [player, label=left:$x_3$]  (c) {};
    \path (63:2.23cm)   node [player, label=right:$y_1$]  (d) {};
    \path (0:1cm)       node [player, label=right:$y_2$]  (e){};
    \path (297:2.23cm)  node [player, label=right:$y_3$]  (f){};
   \draw[black,thick,-stealth] (a) - +(d);
   \draw[black,thick,-stealth] (a) - +(e);
   \draw[black,thick,-stealth] (a) - +(f);
   \draw[black,thick,-stealth] (d) - +(b);
   \draw[black,thick,-stealth] (e) - +(b);
   \draw[black,thick,-stealth] (b) - +(f);
   \draw[black,thick,-stealth] (d) - +(c);
   \draw[black,thick,-stealth] (e) - +(c);
   \draw[black,thick,-stealth] (c) - +(f);
    \end{tikzpicture}

& \phantom{ddddd} &

\begin{tikzpicture}[auto,thick,scale=0.7]
    \tikzstyle{player}=[minimum size=5pt,inner sep=0pt,outer sep=0pt,ball color=black,circle]
    \tikzstyle{source}=[minimum size=5pt,inner sep=0pt,outer sep=0pt,ball color=black, circle]
    \tikzstyle{arc}=[minimum size=5pt,inner sep=1pt,outer sep=1pt, font=\footnotesize]
    \path (117:2.23cm)  node [player, label=left:$x_1$]  (a) {};
    \path (180:1cm)     node [player, label=left:$x_2$]  (b) {};
    \path (243:2.23cm)  node [player, label=left:$x_3$]  (c) {};
    \path (63:2.23cm)   node [player, label=right:$y_1$]  (d) {};
    \path (0:1cm)       node [player, label=right:$y_2$]  (e){};
    \path (297:2.23cm)  node [player, label=right:$y_3$]  (f){};
   \draw[black,thick,-stealth] (a) - +(d);
   \draw[black,thick,-stealth] (e) - +(a);
   \draw[black,thick,-stealth] (a) - +(f);
   \draw[black,thick,-stealth] (d) - +(b);
   \draw[black,thick,-stealth] (b) - +(e);
   \draw[black,thick,-stealth] (b) - +(f);
   \draw[black,thick,-stealth] (d) - +(c);
   \draw[black,thick,-stealth] (e) - +(c);
   \draw[black,thick,-stealth] (c) - +(f);
    \end{tikzpicture}
\\
 $D$ & \phantom{ddd} & $D'$
\end{tabular}
\end{center}
\caption{$W_0 = \{y_3\}$, $W_1=\{x_2, x_3\}$, $W_2=\{y_1, y_2\}$,  $W_3= \{x_1\}$, and $W_4=\emptyset$ for $D$; $W_0=\{y_3\}$, $W_1=\{x_3\}$ and $V(D'_2)=\{x_1, x_2, y_1, y_2\}$ for $D'$. Thus $\zeta(D)=4$ and $\zeta(D')=2$.}\label{fig:sink}
\end{figure}

By definition, it is easy to see that $V(D_{\zeta(D)}) \neq \emptyset$ and $\bigcup_{i=0}^{\zeta(D)-1}{W_i} \cup V(D_{\zeta(D)}) = V(D)$ for a digraph $D$.
Therefore we have the following proposition.

\begin{Prop}\label{prop:union}
  For a digraph $D$, $W_{\zeta(D)}=V(D_{\zeta(D)})$ if and only if $\bigcup_{i=0}^{\zeta(D)}{W_i} = V(D)$.
\end{Prop}

\begin{Prop}\label{prop:acyclic}
  A digraph $D$ is acyclic if and only if $W_{\zeta(D)} \neq \emptyset$.
\end{Prop}

\begin{proof}
   We note that $W_{\zeta(D)} = \emptyset$ if and only if $D_{\zeta(D)}$ has no sinks if and only if $D_{\zeta(D)}$ has a directed cycle.
   Therefore, if $W_{\zeta(D)} = \emptyset$, then $D_{\zeta(D)}$ has a directed cycle and so $D$ has a directed cycle.
   To show the converse, suppose that $D$ has a directed cycle $C$.
    Then any vertex on $C$ cannot belong to $W_i$ for any $i = 0, \ldots, \zeta(D)$, so the vertices of $C$ belong to $V(D_{\zeta(D)})$.
    Thus $W_{\zeta(D)} \neq V(D_{\zeta(D)})$ and hence $W_{\zeta(D)}=\emptyset$ by definition.
\end{proof}
Let $D$ be a bipartite tournament with bipartition $(V_1,V_2)$ and $\zeta(D) \ge 1$, and let $(W_0, \ldots, W_{\zeta(D)})$ and $(D_0, D_1, \ldots, D_{\zeta(D)})$ be the sink sequence and the digraph sequence associated with the sink sequence, respectively, of $D$. Fix $j \in\{0, \ldots, \zeta(D)\}$.
By definition, $D_j$ is a bipartite tournament and $W_j$ is the set of sinks in $D_j$. Therefore $W_j$ is included in exactly one of partite sets of $D_j$. Since the partite sets of $D_j$ are included in $V_1$ and $V_2$, respectively, $W_j$
is included in exactly one of $V_1$ and $V_2$.

Now we state the following proposition.

\begin{Prop}\label{lem:oddeven}
  Let $D$ be a bipartite tournament with bipartition $(V_1,V_2)$. In addition, let $(W_0, \ldots, W_{\zeta(D)})$ be the sink sequence of $D$ with $\zeta(D) \ge 1$.
  Then, $\bigcup_{0 \le i \le {\zeta(D)}/2}{W_{2i}}$ is included in one of the bipartite sets while $\bigcup_{0 \le i \le ({\zeta(D)}-1)/2}{W_{2i+1}}$ is included in the other partite set.
  Furthermore, $D$ is acyclic if and only if $\bigcup_{0 \le i \le {\zeta(D)}/2}{W_{2i}}$ and $\bigcup_{0 \le i \le ({\zeta(D)}-1)/2}{W_{2i+1}}$ themselves are the bipartite sets.
 \end{Prop}

\begin{proof}
Let $(D_0, D_1, \ldots, D_{\zeta(D)})$ be the digraph sequence associated with $(W_0, \ldots, W_{\zeta(D)})$.
Since $\zeta(D) \ge 1$, $\emptyset \subsetneq W_0 \subsetneq V(D)$.
Suppose that $W_j \subset V_1$ for some $j \in \{0, \ldots, \zeta(D)-1\}$.
Since $D$ is a bipartite tournament, $D_j$ is a bipartite tournament and so $W_{j+1} \subset V_2$.
Similarly, if $W_j \subset V_2$ for some $j \in \{0, \ldots, \zeta(D)-1\}$, then $W_{j+1} \subset V_1$.
Thus we have shown that $\bigcup_{0 \le i \le {\zeta(D)}/2}{W_{2i}}$ is included in one of the bipartite sets while $\bigcup_{0 \le i \le ({\zeta(D)}-1)/2}{W_{2i+1}}$ is included in the other partite set.

The ``furthermore'' part may be justified as follows.
Without loss of generality, we may assume that
\begin{equation}\label{partite}
\bigcup_{0 \le i \le ({\zeta(D)}-1)/2}{W_{2i+1}} \subset V_1 \mbox{ and } \bigcup_{0 \le i \le {\zeta(D)}/2}{W_{2i}} \subset V_2.
\end{equation}
Now,
\begin{tabbing}
  \phantom{dd} \= $D$ is acyclic \hskip3.7in \= \\
  $\Leftrightarrow$ \> $W_{\zeta(D)} \neq \emptyset$ \> (by Proposition~\ref{prop:acyclic}) \\
  $\Leftrightarrow$ \> $V(D_{\zeta(D)})=W_{\zeta(D)}$ \> (by definition) \\
  $\Leftrightarrow$ \> $\bigcup_{i=0}^{\zeta(D)}W_i = V(D)$ \> (by Proposition~\ref{prop:union}) \\
  $\Leftrightarrow$ \> $\bigcup_{0 \le i \le ({\zeta(D)}-1)/2}{W_{2i+1}}=V_1$ and $\bigcup_{0 \le i \le {\zeta(D)}/2}{W_{2i}}=V_2$ \> (by \eqref{partite}).
\end{tabbing}
\end{proof}

\begin{Prop}\label{Lem:noprey}
  Let $D$ be a bipartite tournament with $\zeta(D) \ge 1$ and $(W_0, \ldots, W_{\zeta(D)})$ be the sink sequence of $D$.
  Then any directed walk with an initial vertex in $W_i$ has length at most $i$ in $D$ for $i=0, \ldots, \zeta(D)-1$.
  Furthermore, if $D$ is acyclic, then even a directed walk with an initial vertex in $W_{\zeta(D)}$ has length at most $\zeta(D)$ in $D$.
\end{Prop}

\begin{proof}
Fix $i \in \{0,\ldots, \zeta(D)-1\}$ and take a directed walk $W$ in $D$ with an initial vertex $v_i$ in $W_i$.
Let $\alpha$ be the length of $W$.
Then there are $\alpha$ terms after $v_i$ in the sequence of vertices on $W$.
Suppose to the contrary that there exists a vertex $w$ on $W$ that is distinct from $v_i$ and belongs to $V(D_{\zeta(D)})$.
Without loss of generality, we may assume that $w$ is the first vertex on $W$ that belongs to $V(D_{\zeta(D)})$.
Then the vertex right before $w$ on $W$ belongs to $W_t$ for some $t \in \{0,1,\ldots , \zeta(D)\}$, which contradicts the definition of sink sequence.
Therefore each vertex on $W$ belongs to $W_j$ for some $j \in \{0,1,\ldots , \zeta(D)\}$.
Now suppose that there exist two consecutive vertices $w_1$ and $w_2$ on $W$ such that $(w_1,w_2)$ is an arc on $W$ and $w_1 \in W_j$ and $w_2 \in W_k$ for some positive integers $j$ and $k$ satisfying $j \le k$.
Then $(w_1,w_2)$ belongs to $D_j$ by the definition of digraph sequence associated with $(W_0, \ldots, W_{\zeta(D)})$.
However, the vertices in $W_j$ have outdegree zero in $D_j$ and we reach a contradiction.
Therefore, if $(x,y)$ is an arc on $W$ for some vertices $x$ and $y$ on $W$, then $x \in W_j$ and $y \in W_k$ for some positive integers $j$ and $k$ satisfying $j-k \ge 1$.
Therefore, if the terminus belongs to $W_{l}$ for some nonnegative integer $l$, then $i-l \ge \alpha$.
Since $l \ge 0$, we have $\alpha \le i$ and so the first part of the lemma statement is valid.

If $D$ is acyclic, then, by Proposition~\ref{prop:acyclic}, $W_{\zeta(D)} \neq \emptyset$.
We may take a directed walk with an initial vertex in $W_{\zeta(D)}$.
Then, by the above argument, the length of the walk is at most $\zeta(D)$.
\end{proof}

\begin{Prop}\label{prop:clique}
  Let $D$ be a bipartite tournament with  bipartition $(V_1,V_2)$.
  In addition, let $\zeta(D) \ge 1$ and $(W_0, \ldots, W_{\zeta(D)})$ be the sink sequence of $D$, and $\WWW_k=\bigcup_{i=0}^{k-1}W_i$ for $1 \le k \le \zeta(D)+1$.
  Then the following are true:
  \begin{itemize}
  \item[(1)] A vertex in $W_{2k'+1}$ is an out-neighbor of each vertex in $V_2 \setminus \WWW_{2k'+1}$ in $D$ for $k' =0, \ldots, \left\lfloor {{\zeta(D)} / {2} }  \right\rfloor -1$.
  \item[(2)] A vertex in $W_{2k'}$ is an out-neighbor of each vertex in $V_1 \setminus \WWW_{2k'}$ in $D$ for $k' =1, \ldots, \left\lceil {{\zeta(D)} / {2} }  \right\rceil -1$.
  \item[(3)] If $v_i \in W_i$ for each $i=0, \ldots, l$ ($1 \le l < \zeta(D)$), then there exists a directed path $v_l \rightarrow v_{l-1} \rightarrow \cdots \rightarrow v_1 \rightarrow v_0$. Furthermore, if $D$ is acyclic, then there is an arc from any vertex in $W_{\zeta(D)}$ to any vertex in $W_{\zeta(D)-1}$.
  \end{itemize}
\end{Prop}
\begin{proof}
Since $\zeta(D) \ge 1$, $W_0 \neq \emptyset$.
Since $D$ is a bipartite tournament, $W_0 \subset V_1$ or $W_0 \subset V_2$.
Without loss of generality, we may assume $W_0 \subset V_2$.
Then, by Proposition~\ref{lem:oddeven}, $W_i \subset V_1$ for each odd integer $0 \le i \le \zeta(D)$ and $W_i \subset V_2$ for each even integer $0 \le i \le \zeta(D)$.
Take a vertex $v_i \in W_i$ for each $i=0, \ldots, \zeta(D)-1$.
If $D$ is acyclic, then $W_{\zeta(D)} \neq \emptyset$ by Proposition~\ref{prop:acyclic} and we may take $v_{\zeta(D)} \in W_{\zeta(D)}$.
Fix $k' \in \left\{0, \ldots, \left\lfloor {{\zeta(D)} / {2} }  \right\rfloor -1 \right\}$.
Then, since $D$ is a bipartite tournament, $v_{2k'+1}$ is an out-neighbor of each vertex in $V_2 \setminus \WWW_{2k'+1}$ in $D$.
Therefore there exists an arc from $v_{2k'+2j+2}$ to $v_{2k'+1}$ for each nonnegative integer $j$ with $2k'+2j+2 \le \zeta(D)$.
Similarly, for each $k'= 1, \ldots, \left\lceil {{\zeta(D)} / {2} }  \right\rceil -1$, $v_{2k'}$ is an out-neighbor of each vertex in $V_1 \setminus \WWW_{2k'}$ in $D$, so there exists an arc from $v_{2k'+2j+1}$ to $v_{2k'}$ for each nonnegative integer $j$ with $2k'+2j+1 \le \zeta(D)$.
Since $v_i$ was arbitrarily chosen in $W_i$, we have shown (1) and (2).
We also have shown the following:
\begin{itemize}
  \item there exists an arc from each vertex in $W_t$ to each vertex in $W_s$ for positive integers $0 \le s < t < \zeta(D)$ whenever $t-s$ is an odd integer;
  \item if $D$ is acyclic, there exists an arc from each vertex in $W_{\zeta(D)}$ to each vertex in $W_s$ for positive integer $0 \le s < \zeta(D)$ whenever $\zeta(D)-s$ is an odd integer.
\end{itemize}

The statement (3) immediately follows from these two facts.
\end{proof}

\section{A characterization of the $m$-step competition graph of a bipartite tournament}\label{chap:character}

In this section, we completely characterize the $m$-step competition graph of a bipartite tournament for any integer $m \ge 2$.
In addition, we compute the competition index and the competition period of a bipartite tournament.

We first present fundamental properties of $m$-step competition graphs of bipartite tournaments.
 
\begin{Prop}\label{prop:noedge}
For a bipartite tournament $D$ with bipartition $(V_1,V_2)$, there is no edge joining a vertex in $V_1$ and a vertex in $V_2$ in $C^m(D)$ for any positive integer $m$.
\end{Prop}

\begin{proof}
  For a vertex in $V_1$, a vertex in $V_1$ can be only $2k$-step prey for a positive integer $k$ and a vertex in $V_2$ can be only
  $(2k'-1)$-step prey for a positive integer $k'$
   while, for a vertex in $V_2$, a vertex in $V_1$ can be only $(2l-1)$-step prey for a positive integer $l$ and a vertex in $V_2$ can be only
  $2l'$-step prey for a positive integer $l'$.
  Therefore a vertex in $V_1$ and a vertex in $V_2$ cannot have an $m$-step common prey for any positive integer $m$.
\end{proof}

\begin{Prop}\label{prop:existence}
Let $D$ be a bipartite tournament with bipartition $(V_1,V_2)$.
If $C^M(D)$ has an edge for a positive integer $M$, then so does $C^m(D)$ for any positive integer $m \le M$.
\end{Prop}

\begin{proof}
  If $M=1$, then $m=1$ and the statement is trivially true.
  Suppose $M \ge 2$.
  Let $xy$ be an edge in $C^M(D)$.
  Then $x$ and $y$ belong to the same part by Proposition~\ref{prop:noedge}.
  Without loss of generality, we may assume that $x$ and $y$ belong to $V_1$.
  In addition, $x$ and $y$ have an $M$-step common prey $z$ in $D$.
  Then there exist a directed $(x,z)$-walk $P$ and a directed $(y,z)$-walk $Q$ of length $M$ in $D$.
  Let $x_1$ and $y_1$ be the vertices in $D$ such that $(x,x_1)$ and $(y,y_1)$ are the arcs on $P$ and $Q$, respectively.
  If $x_1$ and $y_1$ are distinct,
  then $z$ is an $(M-1)$-step common prey of $x_1$ and $y_1$ and so $x_1$ and $y_1$ are adjacent in $C^{(M-1)}(D)$.
  If $x_1$ and $y_1$ are the same,
  then the vertex immediately following $z$ on $P$ is an $(M-1)$-step common prey of $x$ and $y$ and so $x$ and $y$ are adjacent in $C^{(M-1)}(D)$.
  Therefore there is an edge in $C^{(M-1)}(D)$.
  If $M \ge 3$, then we may repeat this argument to show that there is an edge in $C^{(M-2)}(D)$.
  In this way, we may show that there is an edge in $C^m(D)$ for any positive integer $m \le M$.
\end{proof}

The following corollary is the contrapositive of Proposition~\ref{prop:existence}.

\begin{Cor}\label{cor:existence}
Let $D$ be a bipartite tournament with bipartition $(V_1,V_2)$.
If $C^M(D)$ is an edgeless graph for a positive integer $M$, then so does $C^m(D)$ for any positive integer $m \ge M$.
\end{Cor}

\begin{Prop}\label{prop:adjacent}
 Let $D$ be a bipartite tournament with no sinks.
 If two vertices are adjacent in $C^M(D)$ for a positive integer $M$, then they are also adjacent in $C^m(D)$ for any positive integer $m \ge M$.
\end{Prop}

\begin{proof}
  Let $x$ and $y$ are adjacent in $C^M(D)$.
  Then $x$ and $y$ have an $M$-step common prey $z$ in $D$.
  By the hypothesis, $z$ has an out-neighbor $w$ in $D$.
  Then $w$ is an $(M+1)$-step common prey of $x$ and $y$.
  Hence $x$ and $y$ are adjacent in $C^{(M+1)}(D)$.
  We may repeat this argument to show that $x$ and $y$ are adjacent in $C^{(M+2)}(D)$.
  In this way, we may show that $x$ and $y$ are adjacent in $C^m(D)$ for any positive integer $m \ge M$.
\end{proof}

\begin{Prop}\label{prop:clique2}
  Let $D$ be a bipartite tournament with  bipartition $(V_1,V_2)$.
  In addition, let $\zeta(D) \ge 1$ and $(W_0, \ldots, W_{\zeta(D)})$ be the sink sequence of $D$, and $\WWW_k=\bigcup_{i=0}^{k-1}W_i$ for $1 \le k \le \zeta(D)+1$.
  Then each of $V_1 \setminus \WWW_m$ and $V_2 \setminus \WWW_m$ forms a clique in $C^m(D)$ for a positive integer $m < \zeta(D)$.
\end{Prop}

\begin{proof}

Since $1 \le m < \zeta(D)$, $V_1 \setminus \WWW_m \neq \emptyset$ and
$V_2 \setminus \WWW_m \neq \emptyset$.
 Take a vertex $v$ in  $V_1 \setminus \WWW_m$.
 Then, since $\WWW_{m-1} \subset \WWW_{m}$, $v \in V_1 \setminus \WWW_{m-1}$.
 If $m$ is even, there exists an arc from $v$ to any vertex in $W_m$  by Proposition~\ref{prop:clique}(2) and so any vertex in $W_1$ is an $m$-step prey of $v$ by (3) of the same proposition.
 If $m$ is odd, there exists an arc from $v$ to any vertex in $W_{m-1}$ by Proposition~\ref{prop:clique}(2) since $v \in V_1 \setminus \WWW_{m-1}$, and so any vertex in $W_0$ is an $m$-step prey of $v$.
 Since $v$ is arbitrarily chosen, a vertex in $W_0$ or a vertex in $W_1$ is an $m$-step prey of every vertex in $V_1 \setminus \WWW_m$ depending upon the parity of $m$.
 Thus $V_1 \setminus \WWW_m$ forms a clique in $C^m(D)$.
By applying a similar argument, we may show that $V_2 \setminus \WWW_m$ forms a clique in $C^m(D)$.
\end{proof}

Next, we characterize the $m$-step competition graph of an acyclic bipartite tournament and compute the competition index and the competition period of an acyclic bipartite tournament.
 
For given two graphs $G_1$ and $G_2$, we call the graph having the vertex set $V(G_1) \cup V(G_2)$ and the edge set $E(G_1) \cup E(G_2)$ the \emph{union} of $G_1$ and $G_2$ and denote it by $G_1 \cup G_2$.
Unless otherwise mentioned, $G_1 \cup G_2$ stands for the union of vertex-disjoint graphs $G_1$ and $G_2$.

\begin{Thm}\label{Thm:acyclic}
   Let $D$ be an acyclic bipartite tournament having bipartition $(V_1,V_2)$, $(W_0, \ldots, W_{\zeta(D)})$ be the sink sequence of $D$, and $\WWW_k=\bigcup_{i=0}^{k-1}W_i$ for $1 \le k \le \zeta(D)+1$.
   Then, for a positive integer $m$, the following are true:
   \begin{itemize}
     \item[(1)] $C^m(D)$ is an empty graph if $m>\zeta(D)$;
     \item[(2)] $C^m(D)$ is isomorphic to $K_{|W_{\zeta(D)}|} \cup I_{|\WWW_m|}$ if $m = \zeta(D)$;
     \item[(3)] $C^m(D)$ is isomorphic to $K_{|V_1 \setminus \WWW_m|} \cup K_{|V_2 \setminus \WWW_m|} \cup I_{|\WWW_m|}$ if $m < \zeta(D)$;
     \item[(4)] $\mathrm{cperiod}(D)=1$;
     \item[(5)] $\mathrm{cindex}(D)=\zeta(D)+1$ if $|W_{\zeta(D)}| \ge 2$; otherwise,
  \[
  \mathrm{cindex}(D)=\begin{cases}
                       \zeta(D) & \mbox{if } \zeta(D) = 1 \mbox{ or } |W_{\zeta(D)-1}| \ge 2;  \\
                       \zeta(D)-1 & \mbox{otherwise}.
                     \end{cases}
  \]
   \end{itemize}
\end{Thm}

\begin{proof}
Suppose $\zeta(D) = 0$.
Then, by definition, $W_0 = V(D)$ or $W_0 =\emptyset$.
Since $D$ is a bipartite tournament, $W_0 \neq V(D)$.
Yet, since $D$ is acyclic, $W_0 \neq \emptyset$ by Proposition~\ref{prop:acyclic}.
Therefore we have reached a contradiction.
Therefore $\zeta(D) \ge 1$.
Suppose $m > \zeta(D)$.
Since $D$ is acyclic, $\bigcup_{i=0}^{\zeta(D)}W_i=V(D)$ by the furthermore part of Proposition~\ref{lem:oddeven} and so no vertex in $D$ has an $m$-step prey in $D$ by Proposition~\ref{Lem:noprey}.
Therefore $C^m(D)$ is an empty graph and so the statement (1) is true.
Then, by the definition of competition period, the statement (4) is immediately true.
Now suppose that $m \le \zeta(D)$.
By Proposition~\ref{Lem:noprey}, no vertex in $\WWW_m$ has an $m$-step prey in $D$ and so every vertex in $\WWW_m$ is an isolated vertex in $C^m(D)$.

Suppose $m = \zeta(D)$.
Since $D$ is acyclic, $W_{\zeta(D)} \neq \emptyset$ by Proposition~\ref{prop:acyclic}.
By the furthermore part of Proposition~\ref{lem:oddeven}, $V(D)=W_{\zeta(D)} \cup \WWW_m$.
Then, since $\zeta(D) \ge 1$, $W_{\zeta(D)}$ forms a clique in $C^m(D)$ by Proposition~\ref{prop:clique}.
Since we have shown that every vertex in $\WWW_m$ is an isolated vertex in $C^m(D)$, $C^m(D)$ is isomorphic to $K_{|W_{\zeta(D)}|} \cup I_{|\WWW_m|}$.

Now suppose $m < \zeta(D)$.
Then, by Proposition~\ref{prop:clique2}, each of  $V_1 \setminus \WWW_m$ and $V_2 \setminus \WWW_m$ forms a clique in $C^m(D)$.
Moreover, by Proposition~\ref{prop:noedge}, there is no edge between $K_{|V_1 \setminus \WWW_m|}$ and $K_{|V_2 \setminus \WWW_m|}$.
Since every vertex in $\WWW_m$ is an isolated vertex in $C^m(D)$,
$C^m(D)$ is isomorphic to
 \[
   K_{|V_1 \setminus \WWW_m|} \cup K_{|V_2 \setminus \WWW_m|} \cup I_{|\WWW_m|}.
   \]
Thus the statement (3) is true.

By (1), $\mathrm{cindex}(D) \le \zeta(D)+1$.
In addition, by (1), it is sufficient to show that $C^i(D)$ is not an empty graph in order to prove that $\mathrm{cindex}(D) \ge i+1$ for a positive integer $i$.
If $|W_{\zeta(D)}| \ge 2$, then $C^{\zeta(D)}(D)$ has edges joining each pair of vertices in $W_{\zeta(D)}$ by Proposition~\ref{prop:clique} and so $\mathrm{cindex}(D) = \zeta(D)+1$.
Now suppose $|W_{\zeta(D)}| \le 1$.
Since $D$ is acyclic, $W_{\zeta(D)} \neq \emptyset$ by Proposition~\ref{prop:acyclic}.
Therefore $|W_{\zeta(D)}| = 1$.
If $\zeta(D) = 1$, then $C(D)$ is empty and so $\mathrm{cindex}(D) = 1=\zeta(D)$.
Consider the case $\zeta(D) \ge 2$.
If $|W_{\zeta(D)-1}| \ge 2$, then $C^{\zeta(D)-1}(D)$ has edges joining each pair of vertices in $W_{\zeta(D)-1}$ and so, by the supposition $|W_{\zeta(D)}| =1$ and Proposition~\ref{Lem:noprey}, $\mathrm{cindex}(D) = \zeta(D)$.
Suppose that $|W_{\zeta(D)-1}| = 1$.
Then $C^{\zeta(D)-2}(D)$ has at least one edge joining a vertex in $W_{\zeta(D)}$ and a vertex in $W_{\zeta(D)-2}$ by Proposition~\ref{prop:clique2} and so $\mathrm{cindex}(D) \ge \zeta(D)-1$.
By Proposition~\ref{Lem:noprey}, the vertices in $\WWW_{\zeta(D)-1}$ are isolated in $C^{\zeta(D)-1}(D)$.
By Proposition~\ref{prop:noedge}, any vertex in $W_{\zeta(D)}$ and any vertex in $W_{\zeta(D)-1}$ are not adjacent in $C^{\zeta(D)-1}(D)$.
Thus, by the suppositions $|W_{\zeta(D)}|=|W_{\zeta(D)-1}|=1$, $C^{\zeta(D)-1}(D)$ is an empty graph.
Then, by Corollary~\ref{cor:existence}, $C^{\zeta(D)}(D)$ is an empty graph.
Hence we may conclude that $\mathrm{cindex}(D) = \zeta(D)-1$.
\end{proof}

In the following, we shall characterize the $m$-step competition graph of a bipartite tournament having a directed cycle and compute the competition index and the competition period of a bipartite tournament having a directed cycle.
To do so, we need the following lemmas.

\begin{Lem}\label{lem:atleastone}
  Let $D$ be a bipartite tournament having bipartition $(V_1,V_2)$ without sinks and
  $G$ be the $m$-step competition graph of $D$ for an integer $m \ge 2$.
  Then, $G[V_i]$ is a complete graph or (not necessarily disjoint) union of two complete graphs for each $i=1,2$.
\end{Lem}

\begin{proof}
  By symmetry, it is sufficient to show that the statement is true for $G_1:=G[V_1]$.
  If $G_1$ is complete, then the statement is trivially true.
  Therefore we may assume that $G_1$ is not complete.
  Then there exist two nonadjacent vertices, say $x$ and $y$, in $G_1$.
  Since $D$ has no sink, $x$ and $y$ cannot have a common out-neighbor in $D$ by Proposition~\ref{prop:adjacent}.
  Then, since $D$ is a bipartite tournament,
  \begin{equation}\label{eqn:neighbors}
    N^+_D (x) \subset N^-_D (y) \quad \text{and} \quad N^+_D (y) \subset N^-_D (x).
  \end{equation}
  Let $X$ and $Y$ be the sets defined by
  \[
  X=\{v \in V_1 \mid N^+_D (v) \subset N^+_D(x) \} \quad \text{and} \quad
  Y=\{v \in V_1 \mid N^+_D(v) \subset N^+_D(y) \}.
  \]
  Since $x \in X$ and $y \in Y$, $X \neq \emptyset$ and $Y \neq \emptyset$.
  Then,
  \[
  X \cap Y = \{v \in V_1 \mid N^+_D (v) \subset N^+_D(x) \cap N^+_D(y)\}.
  \]
  Since $N^+_D(v) \neq \emptyset$ for any $v \in V_1$ and $N^+_D(x) \cap N^+_D(y) = \emptyset$, we have $X \cap Y = \emptyset$.
  By \eqref{eqn:neighbors}, every vertex in $X$ (resp.\ $Y$) has $y$ (resp.\ $x$) as a $2$-step prey, so each of $X$ and $Y$ forms a clique in $C^m(D)$ by Proposition~\ref{prop:adjacent}.
  An $m$-step prey of a vertex in $X$ (resp.\ $Y$) is an $m$-step prey of $x$ (resp.\ $y$) by definition.
  Since $x$ and $y$ are nonadjacent in $G_1$, $x$ and $y$ do not have an $m$-step common prey in $D$.
  Therefore any vertex in $X$ and any vertex in $Y$ do not have an $m$-step common prey in $D$ and thus are not adjacent in $G_1$.
  Now, if $X \cup Y =V_1$, then the statement is immediately true.

  Suppose $X \cup Y \neq V_1$.
  Then $Z:= V_1 \setminus (X \cup Y) \neq \emptyset$.
  Take a vertex $z \in Z$.
  By definition, there exist $v$ and $w$ in $N^+_D(z)$ satisfying $v \notin N^+_D(x)$ and $w \notin N^+_D(y)$.
  Since $D$ is a bipartite tournament, $v \in N^-_D(x)$ and $w \in N^-_D(y)$.
  Thus $x$ and $y$ are $2$-step prey of $z$.
  Since $x$ (resp.\ $y$) is a $2$-step prey of every vertex in $Y$ (resp.\ $X$), every vertex in $X \cup Y$ is adjacent to $z$ in $G_1$ by Proposition~\ref{prop:adjacent}.
  Since $z$ was arbitrarily chosen, every vertex in $Z$ is adjacent to every vertex in $X \cup Y$.
  Since every vertex in $Z$ has $x$ as a $2$-step prey, $Z$ forms a clique in $G_1$ by Proposition~\ref{prop:adjacent}.
  Thus we may conclude that each of $Z \cup X$ and $Z \cup Y$ forms a clique in $G_1$.
  As we have shown that any vertex in $X$ and any vertex in $Y$ are not adjacent in $G_1$, $G_1$ is a union of two complete graphs.
  Hence we have shown that $G_1$ is a complete graph or a union of two complete graphs.
  \end{proof}

\begin{Lem}\label{lem:period at least one}
Let $D$ be a bipartite tournament having bipartition $(V_1,V_2)$ without sinks.
Then $\mathrm{cperiod}(D) = 1$ and $\mathrm{cindex}(D) \le 4$.
\end{Lem}

\begin{proof}
  Cho and Kim~\cite{cho2004competition} showed that a digraph without sinks has competition period $1$, so $\mathrm{cperiod}(D) = 1$.
  In the following, we shall show $\mathrm{cindex}(D) \le 4$.
  Suppose that there exist two vertices $x$ and $y$ such that $x$ and $y$ are adjacent in $C^M(D)$ for some positive integer $M$.
  Without loss of generality, we may assume that $x$ and $y$ belong to $V_1$ by Proposition~\ref{prop:noedge}.
  By Proposition~\ref{prop:adjacent}, it is sufficient to show that $x$ and $y$ are adjacent in $C^m(D)$ for some positive integer $m \le 4$.
  Let
  \[
  X=N_D^+(x) \quad \mbox{and} \quad Y=N_D^+(y).
  \]
  By the hypothesis, $X \neq \emptyset$ and $Y \neq \emptyset$.
  If $X \cap Y \neq \emptyset$, then $x$ and $y$ are adjacent in $C(D)$ and we are done.
  Consider the case $X \cap Y = \emptyset$.
  Then, since $D$ is a bipartite tournament,
   \begin{equation*}
    N^+_D (x) \subset N^-_D (y) \quad \text{and} \quad N^+_D (y) \subset N^-_D (x).
  \end{equation*}
  Moreover, since $X \cap Y = \emptyset$, $V_2$ is a disjoint union of the following sets:
\begin{equation*}\label{v_1}
N_D^-(x) \cap N_D^-(y); \quad
N_D^-(x) \cap N_D^+(y)=Y; \quad
N_D^+(x) \cap N_D^-(y)=X.
\end{equation*}
  Suppose to the contrary that, for each $z \in V_1$, $N_D^+(z)=X$ or $N_D^+(z)=Y$ or $N_D^+(z) \supset X \cup Y $.
  Then, $V_1$ is a disjoint union of $Z_1:=\{z \in V_1 \mid N_D^+(z)=X\}$, $Z_2:=\{z \in V_1 \mid N_D^+(z)=Y\}$, and $V_1 \setminus (Z_1 \cup Z_2)$.
  Then  the only possible $m$-step prey of $x$ or $y$ are vertices in $Y$, vertices in $X$, vertices in $Z_1$, or vertices in $Z_2$.
  To see why, suppose there exists a $(u,v)$-directed walk $W$ for some $u \in V$ and $v \in V_1 \setminus (Z_1 \cup Z_2)$ (resp.\ $v \in V_2 \setminus (X \cup Y)$).
  By definition, the vertex $v_1$ right before $v$ on $W$ belongs to $V_2 \setminus (X \cup Y)$ (resp.\ $V_1 \setminus (Z_1 \cup Z_2)$).
  Then, the vertex right before $v_1$ on $W$ belongs to $V_1 \setminus (Z_1 \cup Z_2)$ (resp.\ $V_2 \setminus (X \cup Y)$).
   Therefore any vertex in $V_1 \setminus (Z_1 \cup Z_2)$ or $V_2 \setminus (X \cup Y)$ is reachable only from a vertex in $V_1 \setminus (Z_1 \cup Z_2)$ or $V_2 \setminus (X \cup Y)$.
Thus we have shown that the only possible $m$-step prey of $x$ or $y$ are vertices in $Y$, vertices in $X$, vertices in $Z_1$, or vertices in $Z_2$.
Yet, a vertex in $Y$ (resp.\ $X$) can be only $(4k_1+3)$-step (resp.\ $(4k_1+1)$-step) prey of $x$ while it can be only $(4k_2+1)$-step (resp.\ $(4k_2+3)$-step) prey of $y$, and a vertex in $Z_1$ (resp.\ a vertex in $Z_2$) can be only $4k_3$-step (resp.\ $(4k_3+2)$-step) prey of $x$ while it can be only $(4k_4+2)$-step (resp.\ $4k_4$-step) prey of $y$ for nonnegative integers $k_1$, $k_2$, $k_3$, and $k_4$.
Therefore there are no $m$-step common prey of $x$ and $y$ for any integer $m \ge 1$, and we reach a contradiction.
Thus there exists $z \in V_1$ such that $N_D^+(z) \neq X$, $N_D^+(z) \neq Y$, and $N_D^+(z) \not\supset X \cup Y$.
Hence
\[
\text{$Z:=\{z \in V_1 \mid N_D^+(z) \neq X, N_D^+(z) \neq Y, \text{ and } N_D^+(z) \not\supset X \cup Y \} \neq \emptyset$.}
\]
Suppose that there is $z \in Z$ such that $N_D^+(z) \not\supset X $ and $N_D^+(z) \not\supset Y$.
Then we may take two vertices $w_1 \in X \setminus N_D^+(z)$ and $w_2 \in Y \setminus N_D^+(z)$.
Since $D$ is a bipartite tournament, $(w_1,z) \in A(D)$ and $(w_2,z)\in A(D)$.
Thus $z$ is a $2$-step common prey of $x$ and $y$ in $D$ and we are done.
Now suppose that $N_D^+(z) \supset X$ or $N_D^+(z) \supset Y$ for any $z \in Z$ and fix $z \in Z$.
Without loss of generality, we may assume $N_D^+(z) \supset X$.
Then, since $N_D^+(z) \neq X$ and $N_D^+(z) \not\supset X \cup Y$ by the definition of $Z$, $N_D^+(z) \setminus X \neq \emptyset$ and $N_D^+(z) \not\supset Y$.
Now we may take $w_4 \in N_D^+(z) \setminus X$ and $w_5 \in Y \setminus N_D^+(z)$.
Since $D$ is a bipartite tournament, $(w_4 , x) \in A(D)$ and $(w_5, z)\in A(D)$.
Since $X \neq \emptyset$, we may take $w_3 \in X$.
By the way, since $X \cap Y = \emptyset$, $(w_3, y) \in A(D)$ and $(w_5, x) \in A(D)$.
Thus $x \rightarrow w_3 \rightarrow y \rightarrow w_5 \rightarrow x$ and
$y \rightarrow w_5 \rightarrow z \rightarrow w_4 \rightarrow x$ are directed walks in $D$.
Hence $x$ and $y$ have $x$ as a $4$-step common prey in $D$ and we are done.
\end{proof}

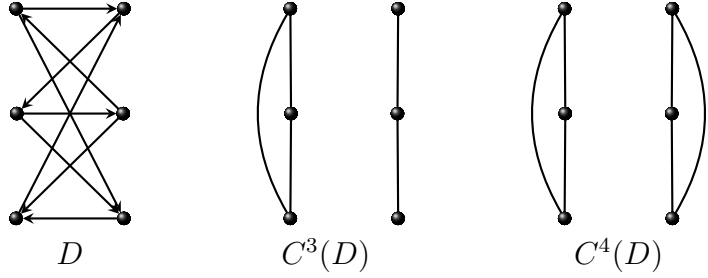
\begin{figure}
  \begin{center}
  \begin{tabular}{ccccc}
    \begin{tikzpicture}[auto,thick,scale=0.7]
    \tikzstyle{player}=[minimum size=5pt,inner sep=0pt,outer sep=0pt,ball color=black,circle]
    \tikzstyle{source}=[minimum size=5pt,inner sep=0pt,outer sep=0pt,ball color=black, circle]
    \tikzstyle{arc}=[minimum size=5pt,inner sep=1pt,outer sep=1pt, font=\footnotesize]
    \path (117:2.23cm)  node [player]  (a){};
    \path (180:1cm)     node [player]  (b) {};
    \path (243:2.23cm)  node [player]  (c) {};
    \path (63:2.23cm)   node [player]  (d) {};
    \path (0:1cm)       node [player]  (e){};
    \path (297:2.23cm)  node [player]  (f){};
   \draw[black,thick,-stealth] (a) - +(d);
   \draw[black,thick,-stealth] (e) - +(a);
   \draw[black,thick,-stealth] (a) - +(f);
   \draw[black,thick,-stealth] (d) - +(b);
   \draw[black,thick,-stealth] (b) - +(e);
   \draw[black,thick,-stealth] (b) - +(f);
   \draw[black,thick,-stealth] (c) - +(d);
   \draw[black,thick,-stealth] (e) - +(c);
   \draw[black,thick,-stealth] (f) - +(c);
    \end{tikzpicture}

& \phantom{ddd} &

\begin{tikzpicture}[auto,thick,scale=0.7]
    \tikzstyle{player}=[minimum size=5pt,inner sep=0pt,outer sep=0pt,ball color=black,circle]
    \tikzstyle{source}=[minimum size=5pt,inner sep=0pt,outer sep=0pt,ball color=black, circle]
    \tikzstyle{arc}=[minimum size=5pt,inner sep=1pt,outer sep=1pt, font=\footnotesize]
    \path (117:2.23cm)  node [player]  (a){};
    \path (180:1cm)     node [player]  (b) {};
    \path (243:2.23cm)  node [player]  (c) {};
    \path (63:2.23cm)   node [player]  (d) {};
    \path (0:1cm)       node [player]  (e){};
    \path (297:2.23cm)  node [player]  (f){};

    \path (a) edge [black, bend right=30,thick,-] (c);
    \draw[black,thick,-] (a) -- ++(b);
    \draw[black,thick,-] (b) -- ++(c);
    \draw[black,thick,-] (d) -- ++(e);
    \draw[black,thick,-] (e) -- ++(f);
    \end{tikzpicture}

& \phantom{ddd} &

\begin{tikzpicture}[auto,thick,scale=0.7]
    \tikzstyle{player}=[minimum size=5pt,inner sep=0pt,outer sep=0pt,ball color=black,circle]
    \tikzstyle{source}=[minimum size=5pt,inner sep=0pt,outer sep=0pt,ball color=black, circle]
    \tikzstyle{arc}=[minimum size=5pt,inner sep=1pt,outer sep=1pt, font=\footnotesize]
    \path (117:2.23cm)  node [player]  (a){};
    \path (180:1cm)     node [player]  (b) {};
    \path (243:2.23cm)  node [player]  (c) {};
    \path (63:2.23cm)   node [player]  (d) {};
    \path (0:1cm)       node [player]  (e){};
    \path (297:2.23cm)  node [player]  (f){};
    \path (a) edge [black, bend right=30,thick,-] (c);
    \path (d) edge [black, bend left=30,thick,-] (f);
    \draw[black,thick,-] (a) -- ++(b);
    \draw[black,thick,-] (b) -- ++(c);
    \draw[black,thick,-] (d) -- ++(e);
    \draw[black,thick,-] (e) -- ++(f);
    \end{tikzpicture}
\\
 $D$ & \phantom{ddd} & $C^3(D)$ & \phantom{ddd} & $C^4(D)$
\end{tabular}
\end{center}
\caption{A digraph $D$ without sinks and with competition index $4$.}\label{fig:index4}
\end{figure}

The upper bound given in Lemma~\ref{lem:period at least one} is sharp as seen by the digraph $D$ given in Figure~\ref{fig:index4}.
The digraph $D$ has no sinks, $C^3(D) \neq C^4(D)$, $C^m(D)=C^4(D)$ for any integer $m \ge 4$. Therefore the competition index of $D$ is $4$.

\begin{Thm}\label{thm:main}
Let $D$ be a bipartite tournament with bipartition $(V_1,V_2)$ which has a directed cycle and $G_m$ be the $m$-step competition graph of $D$ for an integer $m \ge 2$.
Then, $G_m[V_i]$ is a disjoint union of complete graphs among which at most two are nontrivial, or, by deleting the isolated vertices from $G_m[V_i]$ if any, we obtain a non-disjoint union of two complete graphs for each $i=1,2$.
Moreover,
\[
\left\{ \begin{aligned}
\mathrm{cindex}(D)\le 4    &&   \mbox{ and }  && \mathrm{cperiod}(D)=1     && \mbox{ if  }\  \zeta(D) = 0;  \\
\mathrm{cindex}(D)\le 4     &&   \mbox{ and }  && \mathrm{cperiod}(D) \le 2 && \mbox{ if  }\  \zeta(D) = 1;  \\
\mathrm{cindex}(D) =\zeta(D) &&  \mbox{ and }  && \mathrm{cperiod}(D)=1     && \mbox{ if  }\  \zeta(D) \ge 2.
\end{aligned}
 \right.
\]
\end{Thm}

\begin{proof}
Let $(W_0, W_1, \ldots, W_{\zeta(D)})$ be the sink sequence of $D$ and $G_{m,i} = G_m[V_i]$ for $i=1,2$.
Suppose $\zeta(D)=0$.
Then, since $D$ has a directed cycle, $W_0 = \emptyset$ by Proposition~\ref{prop:acyclic}.
Therefore the theorem statement is true by Lemmas~\ref{lem:atleastone} and~\ref{lem:period at least one}.

Now suppose  $\zeta(D) \ge 1$.
By Proposition~\ref{lem:oddeven}, $\bigcup_{0 \le i \le {\zeta(D)}/2}{W_{2i}}$ is included in one of the bipartite sets while $\bigcup_{0 \le i \le ({\zeta(D)}-1)/2}{W_{2i+1}}$ is included in the other partite set.
By symmetry, we may assume that $\bigcup_{0 \le i \le {\zeta(D)}/2}{W_{2i}} \subset V_2$ and $\bigcup_{0 \le i \le ({\zeta(D)}-1)/2}{W_{2i+1}} \subset V_1$.
Let $(D_0, D_1, \ldots, D_{\zeta(D)})$ be the digraph sequence associated with $(W_0, \ldots, W_{\zeta(D)})$.

{\it Case 1.} $\zeta(D) = 1$. Then $(W_0,W_1)$ is the sink sequence of $D$ with $W_0 \subset V_2$ and $W_1 \subset V_1$.
Since  $\zeta(D)=1$, $W_1 = V(D_1)$ or $W_1 = \emptyset$.
By the hypothesis, $D$ has a directed cycle, so $W_1 = \emptyset$ by Proposition~\ref{prop:acyclic}.
Thus every vertex in $D_1$ has outdegree at least one.
Obviously, each vertex in $W_0$ is isolated in $G_m$.
Since $W_1 \neq V(D_1)$,
$W_0 \subsetneq V_2$. Thus $D_1$ is a bipartite tournament with bipartition $(V_1,V_2\setminus W_0)$.

Consider the case where $m$ is odd.
Then, since every vertex in $D_1$ has outdegree at least one, every vertex in $V_1$ has an $(m-1)$-step prey in $V_1$ in $D_1$ and thus in $D$.
Therefore each vertex in $W_0$ is an $m$-step prey of each vertex in $V_1$ in $D$.
Thus $G_{m,1}$ is complete.
Furthermore, since every vertex in $D_1$ has outdegree at least one, the subgraph $H$ of $C^m(D_1)$ induced by $V_2\setminus W_0$ is a complete graph or a union of two complete graphs by Lemma~\ref{lem:atleastone}.
Since the subgraph of $C^m(D)$ induced by $V_2\setminus W_0$ is $G_{m,2} - W_0$, $H$ is a subgraph of $G_{m,2} - W_0$.
If an edge $e$ belongs to $G_{m,2} - W_0$, then the end vertices of $e$ have an $m$-step common prey in $D$ which belongs to $V_1$.
Moreover, since each vertex in $W_0$ has outdegree zero, any directed walk of length $m$ does not contain a vertex in $W_0$ as an interior vertex.
Therefore the adjacency of two vertices belonging to $G_{m,2} - W_0$ is inherited to $H$.
Then, since $H$ is a spanning subgraph of $G_{m,2} - W_0$, $H=G_{m,2} - W_0$.
Thus $G_{m,1}$ is complete, the vertices in $W_0$ are isolated in $G_m$, and $G_{m,2} - W_0$ is a complete graph or a union of two complete graphs.
By applying a similar argument for the case in which $m$ is even, we may show that $G_{m,2} - W_0$ is complete, the vertices in $W_0$ are isolated in $G_m$, and $G_{m,1}$ is a complete graph or a union of two complete graphs.
Hence the first part of the theorem is true.

Now we show $\mathrm{cperiod}(D) \le 2$.
Suppose that $x$ and $y$ are adjacent in $G_M$ for some $x, y \in V_2$ and some odd integer $M \ge 3$.
Then $x$ and $y$ have an $M$-step common prey in $D$.
Since $M$ is odd, the $M$-step common prey of $x$ and $y$ in $D$ are contained in $V_1$.
Since every vertex of $D_1$ has outdegree at least one, $x$ and $y$ have an $(M+2)$-step common prey in $D$.
Thus $x$ and $y$ are adjacent in $G_{M+2}$.
By applying a similar argument, we may show that if $x$ and $y$ are adjacent in $G_M$ for $x, y \in V_1$ for some even integer $M \ge 2$, then $x$ and $y$ are adjacent in $G_{M+2}$.
Thus we have shown that
\begin{itemize}
  \item[($\star$)] if two vertices are adjacent in $G_m$ for $m \ge 2$, then they are adjacent in $G_{m+2}$.
\end{itemize}
Note that $G_{N,1}$ is complete and the vertices in $W_0$ are isolated in $G_N$ for any odd integer $N \ge 3$ and $G_{N',2}-W_0$ is complete and the vertices in $W_0$ are isolated in $G_{N'}$ for any even integer $N' \ge 2$.
Hence $\mathrm{cperiod}(D) \le 2$.

In the following, we compute the competition index of $D$.
In a previous argument, we have shown that for any odd integer $N \ge 3$, $G_{N,1}$ is complete and
for any even integer $N'$, $G_{N',2} - W_0$ is complete.
Now take two vertices $x$ and $y$ in $V_1$ which are adjacent in $G_M$ for some even integer $M \ge 4$.
Then $x$ and $y$ have an $M$-step common prey $z$ in $D$.
Since $M$ is even, $z$ belongs to $V_1$.
By the definition of $W_0$, neither any $(x,z)$-directed walk nor any $(y,z)$-directed walk contains a vertex in $W_0$, so $z$ is an $M$-step common prey of $x$ and $y$ in $D_1$.
Since $D_1$ has no sinks, $\mathrm{cperiod}(D_1) = 1$ and $\mathrm{cindex}(D_1) \le 4$ by Lemma~\ref{lem:period at least one}.
Thus $x$ and $y$ are adjacent in $C^4(D_1)$.
Since $C^4(D_1)$ is a subgraph of $G_4$, $x$ and $y$ are adjacent in $G_4$.
By ($\star$), $x$ and $y$ are adjacent in $G_{N'}$ for any even integer $N' \ge 4$.
By applying a similar argument, we may show that if two vertices in $V_2$ are adjacent in $G_M$ for some odd integer $M \ge 5$, then
they are adjacent in $G_N$ for any odd integer $N \ge 5$.
Hence we have shown that $\mathrm{cindex}(D) \le 4$.

{\it Case 2.} $\zeta(D) \ge 2$.
Let $l=\zeta(D)$.
Since $D$ has a directed cycle, $W_l = \emptyset$ by Proposition~\ref{prop:acyclic} and so each vertex in $D_l$ has outdegree at least one.
Since $D$ contains a directed cycle and $D_l$ is a bipartite tournament, $|V(D_l) \cap V_1| \ge 2$.
 Now we take two distinct vertices $u$ and $v$ in $V(D_l) \cap V_1$.
Since each vertex in $D_l$ has outdegree at least one, $u$ (resp.\ $v$) has an $(m-1)$-step prey $u'$ (resp.\ $v'$) in $V(D_l)$. Obviously, $u'$ and $v'$ belong to $V_1$ if $m$ is odd; to $V_2$ if $m$ is even. If $u'$ and $v'$ belong to $V_1$ (resp.\ $V_2$), then they have a common prey $w$ in $W_0$ (resp.\ $W_1$) by the definition of sink sequence.
Therefore $w$ is an $m$-step common prey of $u$ and $v$ in $D$ and so $V(D_l) \cap V_1$ forms a clique in $G_m$.
By a similar argument, it can be shown that   $V(D_l) \cap V_2$ forms a clique in $G_m$.

For notational convenience, let $\WWW_k=\bigcup_{i=0}^{k-1}W_i$ for $1 \le k \le \zeta(D)+1$.

{\it Subcase 1.} $m \ge l$.
Then, by Proposition~\ref{Lem:noprey}, the vertices in $\WWW_l$ are isolated in $G_m$.
Since we have shown that $V(D_l) \cap V_i$ forms a clique in $G_m$, $G_{m,i}$ is isomorphic to
 \[
   K_{|V(D_l) \cap V_i|} \cup  I_{|V_i \setminus V(D_l)|}
   \]
for each $i=1,2$.

{\it Subcase 2.} $2 \le m \le l-1$.
Then each of $V_1 \setminus \WWW_m$ and $V_2 \setminus \WWW_m$ forms a clique in $G_m$ by Proposition~\ref{prop:clique2}.
By Proposition~\ref{Lem:noprey}, the vertices in $\WWW_m$ are isolated in $G_m$.
Thus $G_{m,i}$ is isomorphic to
 \[
   K_{|V_i \setminus \WWW_m|} \cup  I_{|V_i \cap \WWW_m|}
   \]
for each $i=1,2$.

By the conclusion deduced in Subcase~1, $\mathrm{cperiod}(D) =1$ and $\mathrm{cindex}(D) \le l$.
Suppose that $l$ is even.
Then $W_{l-1} \subset V_1$.
By Proposition~\ref{Lem:noprey}, the vertices in $W_{l-1}$ are isolated in $G_l$.
However, by Proposition~\ref{prop:clique2}, $V_1 \setminus \WWW_{l-1}$ forms a clique in $G_{l-1}$. Since $W_{l-1} \subset \left(V_1 \setminus \WWW_{l-1}\right)$, $W_{l-1}$ forms a clique in $G_{l-1}$.
Thus $G_l \neq G_{l-1}$ and so $\mathrm{cindex}(D) \ge l$.
We may apply similar argument to show that $\mathrm{cindex}(D) \ge l$ for an odd $l$.
Hence $\mathrm{cindex}(D) = l$.
\end{proof}

\section{Acknowledgement}
This research was supported by the National Research Foundation of Korea(NRF) funded by the Korea government(MEST) (NRF-2017R1E1A1A03070489) and by the Korea government(MSIP) (2016R1A5A1008055).



\end{document}